\documentclass[10pt,twoside,english,reqno,a4paper]{amsart}
\usepackage{listings,graphicx,amsmath,varioref,amscd,amssymb,color,bm,stmaryrd,amsthm,amsfonts,graphics,geometry,latexsym,pgf,pst-all} 
\theoremstyle{plain}
\usepackage{esint}
\usepackage{amsthm}

\usepackage{hyperref}
\usepackage{cancel}

\theoremstyle{plain}
\newtheorem{theorem}{Theorem}[section]

\newtheorem{lemma}[theorem]{Lemma}

\newtheorem{corollary}[theorem]{Corollary}

\usepackage{geometry}
\theoremstyle{definition}
\newtheorem{defin}[theorem]{Definition}

\theoremstyle{remark}

\def\bk{\color{black}}

\numberwithin{equation}{section}

\DeclareMathOperator{\R}{\mathbb{R}}

\newcommand{\car}[1]{\raise1pt\hbox{$\chi$}_{#1}}

\def\re{\mathbb{R}}
\def\dys{\displaystyle}

\begin{document}
	
	\title[Existence and non-existence for elliptic equations with singular reactions]{Existence and non-existence phenomena for nonlinear elliptic equations with $L^1$ data and singular reactions}

	\author[F. Oliva]{Francescantonio Oliva}
	\author[F. Petitta]{Francesco Petitta}
	\author[M. F. Stapenhorst]{Matheus F. Stapenhorst}
	\address{Francescantonio Oliva
		\hfill \break\indent
		Dipartimento di Scienze di Base e Applicate per l' Ingegneria, Sapienza Universit\`a di Roma
		\hfill \break\indent
		Via Scarpa 16, 00161 Roma, Italy}
	\email{\tt francescantonio.oliva@uniroma1.it}
	\address{Francesco Petitta
		\hfill \break\indent
		Dipartimento di Scienze di Base e Applicate per l' Ingegneria, Sapienza Universit\`a di Roma
		\hfill \break\indent
		Via Scarpa 16, 00161 Roma, Italy}
	\email{\tt francesco.petitta@uniroma1.it}
	\address{Matheus Stapenhorst
		\hfill \break\indent
		Departamento de Matemática e Computação, Universidade Estadual Paulista-Unesp
		\hfill \break\indent
		R. Roberto S\'imonsen 305, 19060-900 Presidente Prudente-SP, Brazil}
	\email{\tt m.stapenhorst@unesp.br}
\maketitle	
\begin{abstract}
We study existence and non-existence of solutions for singular elliptic boundary value problems as
\begin{equation}\label{eintro}\begin{cases}\tag{1}
		\displaystyle -\Delta_p u+ \frac{a(x)}{u^{\gamma}}=\mu f(x) \ &\text{ in }\Omega,\\
		u>0&\text{ in }\Omega,\\
		u =  0 \ &\text{ on }
		\partial\Omega,
		\end{cases}
\end{equation}
where $\Omega$ is a smooth bounded open subset of $\re^N$ ($N\ge 2$),  $\Delta_p u$ is the $p$-Laplacian with $p>1$,   $0<\gamma\leq 1$,  and $a\geq0$ is bounded and non-trivial.  For any  positive $ f\in L^{1}(\Omega)$ we show that problem \eqref{eintro} is solvable for  any  $\mu >\mu_0>0$, for some $\mu_0$ large enough. As a  reciprocal outcome  we also show that no finite energy solution  exists if $0<\mu<\mu_{0*}$, for some small $\mu_{0*}$. 
 This paper extends the celebrated one of J. I. Diaz, J. M. Morel and L. Oswald (\cite{DMO}) to the case $p\neq2$.  Our result is also  new for $p=2$ provided   the singular term has a critical growth near zero (i.e. $\gamma=1$). 
\end{abstract}

\tableofcontents

\section{Introduction}

Consider the elliptic boundary value  problem 
\begin{equation}\label{introproblem}
	\begin{cases}
		\displaystyle -\Delta_p u+ \frac{a(x)}{u^{\gamma}}=\mu f(x) \ &\text{ in }\Omega,\\
		u>0&\text{ in }\Omega,\\
		u =  0 \ &\text{ on }
		\partial\Omega,
	\end{cases}
\end{equation}
with $p>1$, $\gamma>0$, $a\geq0$ bounded and non-trivial, $f\in L^1(\Omega)$, and $\mu>0$. 
Here $\Omega$ is a smooth bounded open subset of $\re^N$ ($N\ge 2$) and $\Delta_p u= {\rm div} (|\nabla u|^{p-2}\nabla u)$ is the usual $p$-Laplace operator. 

Singular problems as \eqref{introproblem} for $p=2$, i.e. 
\begin{equation}\label{DMOeq}
	\begin{cases}
		\displaystyle -\Delta u+ \frac{a(x)}{u^{\gamma}}=\mu f(x) \ &\text{ in }\Omega,\\
		u>0&\text{ in }\Omega,\\
		u =  0 \ &\text{ on }
		\partial\Omega,
	\end{cases}
\end{equation}
were largely studied in literature both for its applications to non-Newtonian fluids (\cite{nc,fulks}) and for their pure mathematical interest in the context of  singular Lane-Emdem-Fowler semilinear problems. Problem \eqref{DMOeq} in case of  a nonpositive $a(x)$ is nowadays completely understood since the pioneering papers by \cite{crt,stuart,LMcK} (see also   \cite{cc1, cc2, SWL, HSS, BocOrs,OP1,OP}, and \cite{OPs} for a gentle introduction on the subject). 

From a mathematical point of view and when $a(x)$ is nonpositive, the strategy on proving existence in the above literature  mostly relies in constructing a suitable barrier from below for a certain approximating sequence. This procedure applies immediately after an application of the strong maximum principle as the suitable approximating sequence is chosen to be nondecreasing. 
On the other hand, dealing with a nonnegative function $a(x)$ in \eqref{DMOeq} is  more challenging and less studied as the above construction of the barrier does not, in general, simply extend to this case. Here it is  necessary a more sophisticated control from below of the approximating sequence near the boundary of the domain in order to  deal with the singular term $a(x)u^{-\gamma}$. As we will see this barrier from below is essentially constructed as a power of  the \bk first Dirichlet eigenfunction $\varphi_1$  on  $\Omega$.

\medskip 

 The case $a\equiv 1$, $\mu=1$,  and $0<\gamma<1$ was studied in the celebrated paper  \cite{DMO} where existence and non-existence results were shown provided suitable, respectively,   smallness or largeness assumptions are assumed on $f$. Among other results, by a delicate study  of the inverse of $-\Delta$ and the introduction of a convenient weak notion of sub and super-solutions the authors of \cite{DMO} prove in fact that a solution exists if 
 $$
 \int_\Omega f\varphi_1 dx \geq  \overline{C},
 $$
 for a suitable constant $\overline{C}>0$. 
 
 For $\mu> 0$ in \eqref{DMOeq}, this latter condition translates into the fact that there exists $\mu_0$ such that  a solution of  \eqref{DMOeq} does exist provided $\mu>\mu_0$.  Conversely, for $\mu$ small enough a non-existence result also holds\bk;  we remark that \bk a solution  $u$ \bk of  \eqref{DMOeq} is asked in particular  to  satisfy that  ${u^{-\gamma}}$ belongs to $L^{1}(\Omega)$. 
 
 \medskip 
 
  Also, in \cite{CLM}, the authors showed non-existence of classical solutions of \eqref{DMOeq} when $\gamma\geq1$. A problem with sublinear nonlinearity of the form $\mu u^q$ was also studied (\cite{Z}).  More \bk recently, in \cite{PRR}, the case of  a generic Caratheodory function $f(x,u)$ as source  was considered.\bk 

\medskip 
Let us also mention that  there are contexts in which non-strictly positive solutions are obtained, thus leading to the existence of a free-boundary. In \cite{DavMont} (see also \cite{ap}) for instance, the authors considered problems of the type
\begin{equation}\label{Monteq}
	\begin{cases}
		\displaystyle -\Delta u=\left(-\frac{1}{u^{\gamma}}+\lambda f(x,u)\right)\chi_{\{u>0\}} \ &\text{ in }\Omega,\\
		u =  0 \ &\text{ on }
		\partial\Omega.
	\end{cases}
\end{equation}
With suitable assumptions on $f$, for $\lambda$ large, the authors proved existence of a strictly positive solution $u_{\lambda}$. When $\lambda$ is small it was shown that the set $\{u_{\lambda}=0\}$ has positive measure which gives rise to a non-trivial free-boundary. In the context of reaction-diffusion problems, this phenomenon is associated with the existence of a region in which no reaction takes place. See \cite{Aris} and \cite{DiazBook} for further  details. 

\medskip 
Problem \eqref{Monteq} with $f(x,s)=s^{\theta}$, $0<\theta<1$, was considered in \cite{MS}. Using a smooth positive subsolution, the authors obtained a positive solution for large   $\lambda$'s. See also \cite{DavMontConc} and \cite{GR}.  Problems as \eqref{DMOeq} in the presence of \bk gradient terms were also considered in  the \bk literature (see for instance \cite{B, a6, GR2} and references therein for further details).

\medskip
 Turning back to the case $p\not=2$, problems of the form
\begin{equation*}\label{LMeq}
	\begin{cases}
		\displaystyle -\Delta_p u+ \frac{a(x)}{u^{\gamma}}=f(x,u) \ &\text{ in }\Omega,\\
		u>0&\text{ in }\Omega,\\
		u =  0 \ &\text{ on }
		\partial\Omega,
	\end{cases}
\end{equation*}
have been extensively studied in the literature as we mentioned in case $a(x)\leq 0$ (see also \cite{CST,PRR2, PW,OPs}).

\medskip 

In this paper our aim is  to prove \bk existence and non-existence results for problem \eqref{introproblem}, extending, as a consequence, \bk  the result of  \cite{DMO} for $p=2$,  $0<\gamma<1$, and $a\equiv 1$. The  existence result \bk will follow after a suitable sub- and supersolution argument.

Let us stress that we are also able to extend our result to the critical case  $\gamma=1$ under suitable compatibility assumptions on the data.

\medskip 
 As we already mentioned, singular problems naturally appear in the study of pseudoplastic non-Newtonian fluids. The particular case of a reaction term  in \eqref{DMOeq} (i.e. $a(x)\geq 0$)  is of particular interest in reaction-diffusion problems for permeable catalysts (see \cite{Aris}). There are also applications in Micro-Electromechanical Systems (MEMS), where the parameter $\mu$ denotes the applied voltage and the function $f$ plays the role of the varying dielectric permittivity of the elastic membrane $\Omega$. See \cite{MEMS} for more details.
 
 \medskip

 The plan of the paper is  as follows: in Section \ref{sec:prel} we set the notations and some preliminaries which will be useful in the sequel. In Section \ref{sec:main} we state the main assumptions and existence result in case $0<\gamma<1$. We also give a non-existence result provided $\mu$ is small enough.
 In Section \ref{sec_ex<1} we prove the results presented in Section \ref{sec:main}. Finally, in Section \ref{sec:gamma=1} we  discuss the critical case $\gamma=1$.
 
\subsection{Notation and preliminaries}
\label{sec:prel}

In the entire paper $\Omega\subset\mathbb{R}^{N}$ ($N\ge 2$) denotes an open bounded set with smooth boundary. By $|\Omega|$ we denote the size of $\Omega$ with respect to the $N$-dimensional Lebesgue measure.

The distance of a point $x\in \Omega$ from the boundary of $\Omega$ is denoted by  
$$\delta(x):= \operatorname{dist}(x, \partial \Omega),$$
and it is frequently used that, as $\partial \Omega$ is smooth, it holds
\begin{equation}\label{distint}
\int_{\Omega} \delta^r (x) dx<\infty \ \text{if and only if }r>-1.
\end{equation}
Let us stress that in order to avoid technicalities, without loosing generality and with a little  abuse of notation, we will refer to $\delta(x)$ as a suitable positive smooth (say $C^{1}$) modification of the distance function which agrees with  $\delta (x)$ in a fixed neighbourhood of $\partial \Omega$.  

In particular, for $\varepsilon >0$,  we denote by 
\begin{equation}\label{bordo} 
	\Omega_\varepsilon:=\{x\in \Omega: \delta(x)<\varepsilon\}
\end{equation}	 
an $\varepsilon$-neighbourhood of $\partial\Omega$. 

\begin{figure}[htbp]\centering
	\includegraphics[width=3in]{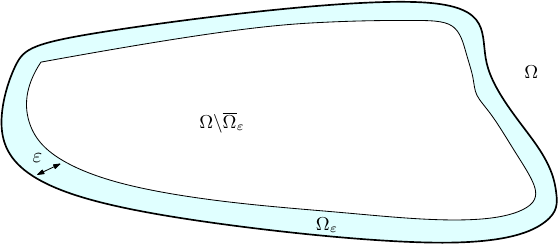}
	\caption{The $\varepsilon$-neighborhood of $\partial\Omega$}\label{oe}
\end{figure}

In the sequel we make use of the standard truncation function defined, for any fixed $k>0$, as  
$$T_k(s)=\max (-k,\min (s,k)),$$
for any   $s\in\re$.

As we will deal with measurable functions not necessarily belonging to $W^{1,1}(\Omega)$\bk, we will use the following result which clarifies the concept of weak derivative for functions having the truncation at any level belonging to a Sobolev space. 
\begin{lemma}\cite[Lemma 2.1]{B1995}\label{l.bocc}
	Let $u:\Omega\to\mathbb{R}$ be a measurable function such that $T_{k}(u)\in W^{1,1}_{loc}(\Omega)$ for every $k>0$. Then, there exists a measurable vectorial  function $v:\Omega\to\mathbb{R}^{N}$ such that
	$$
	\nabla T_{k}(u)=v\chi_{\{|u|\leq k\}}\text{ for a.e. }x\in\Omega\text{ and for every }k>0.
	$$
	Furthermore, $u\in W^{1,1}_{loc}(\Omega)$ if and only if  $v\in L^{1}_{loc}(\Omega)$ \bk and then $v=\nabla u$ in the usual distributional sense.
\end{lemma}

From now on, for a measurable function $u$ such that $T_k(u)\in W^{1,1}_{\rm loc}(\Omega)$ for every $k>0$, with $\nabla u$ we mean the function $v$ defined in Lemma \ref{l.bocc}.

Let  us \bk also recall that a function $u: \Omega \to \R$ belongs to the  Marcinkiewicz space (also called {\it weak Lebesgue space}) $M^s(\Omega)$, with $s>0$, if there exists a positive constant $c$ such that
\begin{equation*}
	\label{marc-eq}
	|\{x\in \Omega : |u(x)| > \lambda\}|\leq \frac{c}{\lambda^s} \qquad \forall \lambda >0.
\end{equation*}
It is known that $L^1(\Omega)\subset M^1(\Omega)$ and $L^{s}(\Omega)\subset M^{s}(\Omega) \subset L^{s-\varepsilon}(\Omega)$ for every $s>1$ and $0<\varepsilon \leq s-1$.

\medskip

We will use the following classical comparison principle:

\begin{lemma}{\cite[Theorem 2.15]{Lindq}}\label{lem comparison}
	Let $u_{1},u_{2}\in W_{0}^{1,p}(\Omega)$ be such that
	\begin{equation*}\label{comp ineq}
		\int_{\Omega}|\nabla u_{1}|^{p-2}\nabla u_{1}\cdot \nabla v\leq \int_{\Omega}|\nabla u_{2}|^{p-2}\nabla u_{2}\cdot\nabla v \ \text{ for all }v\in W_{0}^{1,p}(\Omega),  v\geq0.
	\end{equation*}
	Then $u_{1}\leq u_{2}$  a.e. \bk in $\Omega$.
\end{lemma}

Let us consider the Rayleigh quotient
\begin{equation}\label{e.rayleigh}
	\dys \lambda_{p}=\min_{u\in W_{0}^{1,p}(\Omega),u\neq0}\frac{\int_{\Omega}|\nabla u|^p}{\int_{\Omega}|u|^{p}}, 
\end{equation}
which is the Dirichlet eigenvalue problem associated with the $p$-Laplacian. It is known that the ${\rm argmin} (\lambda_p)$  of \eqref{e.rayleigh}  is the first eigenfunction  $0<\varphi_{1}\in C^{1}(\overline{\Omega})$ which solves
\begin{equation}\label{e.autofunc}
	\begin{cases}
		\displaystyle -\Delta_p \varphi_{1}=\lambda_{p}\varphi_{1}^{p-1} & \text{ in }\Omega,\\
		\varphi_{1} =  0 &\text{ on }
		\partial\Omega.
	\end{cases}
\end{equation}

  \smallskip 

Also observe that, as an easy consequence of the Hopf Lemma, there exist constants $c_{1},c_{2}>0$ such that
\begin{equation}\label{phi equiv dist }
c_{1}\delta(x)\leq \varphi_{1}(x)\leq c_{2}\delta(x)\text{ for all }x\in\Omega.
\end{equation}

Finally,  for the sake of simplicity,  we use the simplified notation $$\int_{\Omega} f:=\int_{\Omega} f(x)dx,$$ when referring to integrals when no ambiguity is possible.

\section{Main assumptions and results for $0<\gamma<1$}
\label{sec:main}

Here we address the existence and non-existence of nonnegative  solutions to
\begin{equation}\label{e.problem}
\begin{cases}
	\displaystyle -\Delta_p u+ \frac{a(x)}{u^{\gamma}}=\mu f(x) \ &\text{ in }\Omega,\\
	u =  0 \ &\text{ on }
	\partial\Omega,
\end{cases}
\end{equation}
where $p>1$, $0<\gamma  <  1$, $\mu>0$, $0\lneq a\in L^{\infty}(\Omega)$ and $f\in L^1(\Omega)$ satisfies 
\begin{equation}\label{e.hypf1}
	\forall \ \omega \subset\subset \Omega\  \ \exists\  c_\omega>0: \ f(x)\geq c_\omega \ \ \text{for a.e. }\ x\in\omega.
\end{equation}
 This condition ensures that $f$ is a.e. bounded away from zero on any set that is compactly embedded in $\Omega$. Yet, it allows for $f$ to be equal to zero on $\partial\Omega$. \bk 
We  now \bk precise what we mean by a solution to \eqref{e.problem}.

\begin{defin}\label{def}
	A nonnegative measurable function $u$ is a distributional solution of \eqref{e.problem} if $T_k(u)\in W^{1,1}_0(\Omega)$ for every $k>0$, if both $|\nabla u|^{p-1}$ and  $au^{-\gamma}$ belong to  $L^1_{\rm loc}(\Omega)$,  and 
	\begin{equation}\label{e.defweaksol}
		\int_{\Omega}|\nabla u|^{p-2}\nabla u\cdot \nabla\varphi+\int_{\Omega}\frac{a\varphi}{u^{\gamma}}=\mu\int_{\Omega}f\varphi\ \ \text{ for all }\varphi\in C_{c}^{1}(\Omega).
	\end{equation}
	
\end{defin}

We underline that in the sequel $u$ is meant to be a \textit{finite energy solution} to \eqref{e.problem} if $u$ is a solution in the sense of Definition \ref{def} and it belongs to $W^{1,p}_0(\Omega)$; otherwise we may  refer to it as an \textit{infinite energy solution} or simply as a solution in the sense of Definition \ref{def}.

\medskip

Let us state our  main existence theorem.

\begin{theorem}\label{teo_existence}
	Let $ 1<p<N\bk$, $0<\gamma<1$, $0\lneq a\in L^{\infty}(\Omega)$ and let $f\in L^q(\Omega)$ be a function satisfying \eqref{e.hypf1}. Then there exists a positive $\mu_0$ such that for any $\mu\ge \mu_0$ the following holds:  
	\begin{enumerate}
		\item[i)] if $q=\frac{Np}{Np-N+p}$ then there exists a positive finite energy solution $u$ to \eqref{e.problem};
		\item[ii)]  if $q=1$ then there exists a positive infinite energy solution $u$ to \eqref{e.problem} such that $u\in M^{\frac{N(p-1)}{N-p}}(\Omega)$ and $|\nabla u|\in M^{\frac{N(p-1)}{N-1}}(\Omega)$.
	\end{enumerate}
Furthermore $au^{-\gamma}$ belongs to $L^1(\Omega)$,   and there exists $t_0>0$, depending only on $\Omega, a,  \gamma$ and $p$,  such that $$u\geq t_0\varphi_1^{\frac{p}{p-1+\gamma}} \ \text{a.e. in }\ \ \Omega,$$ where $\varphi_1$ is given by \eqref{e.autofunc}. \bk
\end{theorem} 
As one might expect, the case $p\geq N$ is simpler. In this situation, the result can be stated as follows:
\begin{theorem}\label{theo pgeqN}
Let $p\geq N$,	$0<\gamma<1$, $0\lneq a\in L^{\infty}(\Omega)$ and let $f\in L^q(\Omega)$ be a function satisfying \eqref{e.hypf1}. Then there exists a positive $\mu_0$ such that for any $\mu\ge \mu_0$ the following holds:  
\begin{enumerate}
	\item[i)] if $p=N$ and $q>1$, then there exists a positive finite energy solution $u$ to \eqref{e.problem};
	\item[ii)] if $p>N$ and $q=1$, then there exists a positive finite energy solution $u$ to \eqref{e.problem}.
	\end{enumerate}
	Furthermore $au^{-\gamma}$ belongs to $L^1(\Omega)$,   and there exists $t_0>0$, depending only on $\Omega, a,  \gamma$ and $p$,  such that $$u\geq t_0\varphi_1^{\frac{p}{p-1+\gamma}} \ \text{a.e. in }\ \ \Omega,$$ where $\varphi_1$ is given by \eqref{e.autofunc}. \bk
\end{theorem}
\bk
\medskip

As a counterpart result of the previous theorem we   show that no finite energy solution exists for a small enough $\mu$. In order to state our result we need a slightly stronger assumption on $a$, i.e.  when $\mu$ is small enough, \bk we request that $a$ is bounded away from zero; on the other hand the result holds true for a generic $a\in L^1(\Omega)$. 

\begin{theorem}\label{t.nonexistence}
	Let $0<c_0 \le a\in L^1(\Omega)$ for some positive constant $c_0$,  and let $0\le f\in L^{\infty}(\Omega)$. Then there exists $\mu_{0*}>0$ such that problem \eqref{e.problem} does not admit a finite energy solution  for any $0<\mu<\mu_{0*}$.
\end{theorem}
The following regularity result, which is interesting by itself,  will be useful in the sequel.  Firstly, it  shows that if a solution has finite energy and the datum is regular enough then the set of test functions in \eqref{e.defweaksol} can be suitably enlarged. Furthermore,   if the datum $f$ is summable enough, then any solution to \eqref{e.problem} is actually  essentially \bk bounded.

\begin{lemma}\label{Lemma_reg}
	Let $1<p<N\bk, \gamma>0$, let $0\le a\in L^1(\Omega)$, $\mu>0$,  and $0\le f\in L^{\frac{Np}{Np-N+p}}(\Omega)$. Moreover let $u\in W^{1,p}_0(\Omega)$ be a solution to \eqref{e.problem} in sense of Definition \ref{def}. Then $$\frac{a} {u^{\gamma}}v \in L^1(\Omega)\qquad \text{for every $v \in W^{1,p}_0(\Omega)$}.$$ Moreover it holds 
\begin{equation}\label{e.c2}
	\int_{\Omega}|\nabla u|^{p-2}\nabla u\cdot \nabla v+	\int_{\Omega}\frac{a} {u^{\gamma}}v = \mu\int_{\Omega}fv\qquad  \text{ for every }v\in W_{0}^{1,p}(\Omega).
\end{equation}
Finally, if $f\in L^{r}(\Omega)$ for some $r>\frac{N}{p}$, then $u\in L^{\infty}(\Omega)$.
\end{lemma}
\begin{proof}
	Let us prove the result for a nonnegative $v\in W_{0}^{1,p}(\Omega)$; the case of $v$ with general sign will easily follow. 
	
	\medskip
	
	Consider at first a sequence  $\varphi_{n} \in C^{1}_{c}(\Omega)$ such that $\varphi_{n} \to v$ strongly in $W^{1,p}_{0}(\Omega)$ as $n\to\infty$. Now let $\rho_{\eta}$ be a standard mollifier. It holds that $\psi_{n,\eta}=\rho_{\eta}* \min\{v, \varphi_n\} \in C^{1}_{c}(\Omega)$ for $\eta>0$ small enough. Then we test \eqref{e.defweaksol} with $\varphi=\psi_{n,\eta}$, i.e. 
	\begin{equation}\label{rem_ex}
		\int_{\Omega} |\nabla u|^{p-2}\nabla u\cdot \nabla \psi_{n,\eta}+ \int_{\Omega} \frac{a \psi_{n,\eta}}{u^\gamma}= \mu\int_{\Omega} f \psi_{n,\eta}.
	\end{equation}
	Now observe that $\psi_{n,\eta} \to \psi_n=\min\{\varphi_{n},v\}$ as $\eta \to 0$ strongly in $W^{1,p}(\Omega)$, *-weak in $L^{\infty}(\Omega)$ and it has compact support in $\Omega$. Then, since  $|\nabla u|^{p-1}\in L^{\frac{p}{p-1}}(\Omega)$, $au^{-\gamma} \in L^1_{\rm loc}(\Omega)$,  and $f\in L^{\frac{Np}{Np-N+p}}(\Omega)$, we may let  $\eta \to 0$ in \eqref{rem_ex} in order to deduce that   
	\begin{equation}\label{rem_ex2}
		\int_{\Omega} |\nabla u|^{p-2}\nabla u\cdot \nabla \psi_{n}+ \int_{\Omega} \frac{a\psi_{n}}{u^\gamma}= \mu\int_{\Omega} f \psi_{n}.
	\end{equation}
	
	Therefore, since $\psi_{n} \to v$ strongly in $W^{1,p}_0(\Omega)$ as $n\to\infty$, we can take $n\to\infty$ in the first and in the third term of \eqref{rem_ex2} (observe that $\psi_{n} \to v$ also weakly in $L^{\frac{Np}{N-p}}(\Omega)$). Hence, as up to subsequences $\psi_{n} \to v$  a.e. \bk in $\Omega$, an application of the Fatou Lemma gives that
	$$\int_{\Omega} \frac{av}{u^\gamma} \le \liminf_{n \to \infty} \int_{\Omega} \frac{a\psi_n}{u^\gamma} =   -  \int_{\Omega} |\nabla u|^{p-2}\nabla u \cdot \nabla v + \mu\int_{\Omega} f v,$$
	which implies  $avu^{-\gamma} \in L^{1}(\Omega)$. Hence, as $0\leq \psi_n\leq v$,  after an application of the Lebesgue Theorem in the second term of \eqref{rem_ex2} one can deduce that \eqref{e.c2} holds.
	
	\medskip
	
	 We finally show that $u$ is  essentially \bk bounded if $f\in L^{r}(\Omega)$ with $r>\frac{N}{p}$. Indeed, let $w\in W_{0}^{1,p}(\Omega)$ be the solution to $-\Delta_{p}w=\mu f$. The  classical Stampacchia   argument gives that $w$ is  essentially \bk bounded. Furthermore, from \eqref{e.c2} one gets
	$$
	\int_{\Omega}|\nabla u|^{p-2}\nabla u\cdot\nabla v\leq\mu\int_{\Omega}fv=\int_{\Omega}|\nabla w|^{p-2}\nabla w\cdot\nabla v\qquad \text{ for all }v\in W_{0}^{1,p}(\Omega), v\geq0,
	$$
	which allows to apply Lemma \ref{lem comparison} in order to deduce that $u\leq w$  a.e. \bk in $\Omega$. This concludes the proof. 
\end{proof}

\section{Proof of existence and non-existence results}
\label{sec_ex<1}

In order to prove  Theorems \ref{teo_existence} and \ref{theo pgeqN} \bk we perform an iteration procedure.  Before introducing the scheme, it is  useful to make some preliminary comments on the strategy. As we will see,  the approximating solutions are shown to be bounded from below by a suitable power of the first eigenfunction of the $p$-Laplacian in $\Omega$ (defined in \eqref{e.autofunc}), namely $t_0\varphi_1^r$ with $r= \frac{p}{p+\gamma-1}>1$ for some positive constant $t_0$ which will be fixed later; roughly speaking,  we will consider  $t_0\|\varphi_1^r\|_{L^\infty(\Omega)}$ as the first step of the iteration and from it we will construct a suitable sequence of functions $u_n$ that, at the end,  will converge to a solution of problem \eqref{e.problem}. 

\smallskip 
Let us recall that for $\varepsilon >0$ small enough we denote by $\Omega_\varepsilon:=\{x\in \Omega: \delta(x)<\varepsilon\}$. We will frequently consider $\Omega_{\varepsilon}$ with $\varepsilon\le \overline{\varepsilon}$ for some fixed $\overline{\varepsilon}>0$.

\medskip

We now  introduce the iterative scheme; let $u_{0}=t_0\|\varphi_1^r\|_{L^\infty(\Omega)}\text{ in }\Omega$, where $t_0>0$  will be fixed in Lemma \ref{lemma_controllobasso} below,  and,  for   $n>0$, consider  $u_n$ as the (weak) solution of 
\begin{equation}\label{pbapprox}
\begin{cases}
		\displaystyle -\Delta_p u_{n}+ \frac{a(x)}{\left(|u_{n-1}|+\frac{1}{n}\right)^{\gamma}}=\mu f_{n}(x)&\text{ in }\Omega,\\
		u_{n} =  0 &\text{ on }
		\partial\Omega,
\end{cases}
\end{equation}
where $f_n=T_{n+c_{\Omega\setminus\overline{\Omega_{\overline{\varepsilon}}}}}(f)$, $r= \frac{p}{p+\gamma-1}$. Here $c_{\Omega\setminus\overline{\Omega_{\overline{\varepsilon}}}}$ is defined by  \eqref{e.hypf1} (with $\omega=\Omega\setminus\overline{\Omega_{\overline{\varepsilon}}} $); in particular, as far as $f_n$ is defined  one   deduces that 
\begin{equation}\label{condfn}
	f_n(x) \ge c_{\Omega\setminus\overline{\Omega_{\overline{\varepsilon}}} } \ge c_{ \Omega\setminus\overline{\Omega_{\varepsilon}} } \text{ for a.e. } x \in  \Omega\setminus\overline{\Omega_{\overline{\varepsilon}}} ,\qquad  \forall \varepsilon<\overline{\varepsilon}, 
\end{equation} 
since  $c_{\Omega\setminus\overline{\Omega_{\overline{\varepsilon}}}}$, defined again in  \eqref{e.hypf1} (with $\omega=\Omega\setminus\overline{\Omega_{\overline{\varepsilon}}} $), is non-increasing as $\varepsilon$ decreases. 

The existence of a function $u_{n} \in W^{1,p}_0(\Omega)\cap L^\infty(\Omega)$ solution to \eqref{pbapprox} follows by standard Leray-Lions Theorem (\cite{ll}). 

\medskip

We start proving a comparison lemma, which provides a barrier from below for $u_n$. We explicitly stress that the constant $t_0$ appearing in the next result does not depend on $n$ and this will be crucial in the proof of Theorem \ref{teo_existence} in order to  pass to the limit in the singular term. 

\begin{lemma}\label{lemma_controllobasso}
	Let $p>1$, $0<\gamma<1$, $0\lneq a\in L^{\infty}(\Omega)$, and let $f\in L^1(\Omega)$  satisfy \bk \eqref{e.hypf1}. Let $u_{n}$ be a solution to \eqref{pbapprox}. Then there exist positive numbers $\mu_0$ and $t_0$ such that
\begin{equation}\label{comparison}
	u_{n}\geq  t_0 \varphi_1^{\frac{p}{p+\gamma-1}} \quad\text{ for all } \mu\geq\mu_0,\quad n\in\mathbb{N},
\end{equation}
	where $\mu_0$ and $t_0$ depend only on $\Omega, a, f, \gamma$ and $p$.
\end{lemma}

\begin{proof}
First we show that $v=t  \varphi_{1}^{r}$, where $r= \frac{p}{p+\gamma-1}>1$ and $t>0$ is large enough, satisfies  
\begin{equation}\label{subsol0}
	-\Delta_p v+\frac{a}{(v+\frac{1}{n})^{\gamma}}\leq \mu f_n \ \ \text{ in }\Omega,
\end{equation} 
for any $n\in\mathbb{N}$,  and for any $\mu\geq \mu_0$ for some suitably large $\mu_0$.
To this aim, a simple computation takes to
$$
-\Delta_p v=-(t r)^{p-1}(r-1)(p-1)\varphi_{1}^{(r-1)(p-1)-1}|\nabla\varphi_{1}|^{p}-(t r)^{p-1}\varphi_{1}^{(r-1)(p-1)}\Delta_{p}\varphi_{1},
$$
which implies {that} \bk
$$
-\Delta_p v=-\frac{t^{p+\gamma-1}}{v^{\gamma}}\left(r^{p-1}(r-1)(p-1)\varphi_{1}^{(r-1)(p-1)-1+\gamma r}|\nabla\varphi_{1}|^{p}+ r^{p-1}\varphi_{1}^{(r-1)(p-1)+\gamma r}\Delta_{p}\varphi_{1}\right).
$$
{ Using \eqref{e.autofunc} and the fact that $r=\frac{p}{p+\gamma-1}$, one gets} \bk
$$
-\Delta_p v=\frac{t^{p+\gamma-1}}{v^{\gamma}}\left(r^{p-1}(1-r)(p-1)|\nabla\varphi_{1}|^{p}+ r^{p-1}\lambda_{p}\varphi_{1}^{p}\right).
$$

Hence we have shown that 
\begin{equation}\label{sub1}
-\Delta_p v=\frac{t^{p+\gamma -1}}{v^{\gamma}}\left(-C_{p,\gamma}|\nabla\varphi_{1}|^{p}+ D_{p,\gamma,\Omega}\varphi_{1}^{p}\right)=-t^{p-1} C_{p,\gamma}|\nabla\varphi_{1}|^{p}\varphi_{1}^{-r\gamma}+t^{p-1}D_{p,\gamma,\Omega}\varphi_{1}^{p-r\gamma},
\end{equation}
where
$$
C_{p,\gamma}=r^{p-1}(r-1)(p-1) \ \text{ and } \ D_{p,\gamma,\Omega}=r^{p-1}\lambda_{p}.
$$
{Then, it  follows from \eqref{sub1} that}\bk
\begin{equation}\label{e.subsolformula}
-\Delta_p v + \frac{a(x)}{v^\gamma} = -t^{p-1} C_{p,\gamma}|\nabla\varphi_{1}|^{p}\varphi_{1}^{-r\gamma}+t^{p-1}D_{p,\gamma,\Omega}\varphi_{1}^{p-r\gamma}+a(x)t^{-\gamma}\varphi_{1}^{-r\gamma}.
\end{equation}
As we aim to prove the validity of \eqref{subsol0}, we need to show that  the right-hand of \eqref{e.subsolformula} does not exceed $\mu f_{n}$ for $\mu$ large enough. Indeed this will be sufficient by simply observing that $$a(x)\left(v+{\frac{1}{n}}\right)^{-\gamma} \le a(x)v^{-\gamma}\qquad \text{in}\quad\Omega.$$

We proceed by estimating the right-hand of \eqref{e.subsolformula} splitting the calculation both  near the boundary of $\Omega$ (i.e. in $\Omega_\varepsilon$) and in  the interior $\Omega$ (i.e. in $\Omega\backslash\overline{\Omega_\varepsilon}$). 

\medskip 
 We first estimate \eqref{e.subsolformula} into $ \Omega\setminus\overline{\Omega_{\overline{\varepsilon}}}$ ($\overline{\varepsilon}$ to be fixed later). 
We get rid of the first negative term on the right-hand of \eqref{e.subsolformula} and we ask for 
$$a(x)t^{-\gamma}\varphi_{1}^{-r\gamma} \le t^{p-1}D_{p,\gamma,\Omega}\varphi_{1}^{p-r\gamma},$$
for any $x\in\Omega\setminus\overline{\Omega_{\overline{\varepsilon}}}$ which means requiring
\begin{equation}\label{tzero}
	t\ge t_0:= \left(\frac{\|a\|_{L^\infty(\Omega)}}{\displaystyle D_{p,\gamma,\Omega}\min_{\Omega\setminus\overline{\Omega_{\overline{\varepsilon}}} }(\varphi_{1}^{p})}\right)^{\frac{1}{p+\gamma-1}}.
\end{equation}
Therefore we have shown that 
\begin{equation}\label{subdentro}
	-\Delta_p v +\frac{a(x)}{v^{\gamma}} \leq 2t^{p-1}D_{p,\gamma, \Omega}\varphi_{1}^{p-r\gamma} \text{ in }\Omega\setminus\overline{\Omega_{\overline{\varepsilon}}} \ \text{for} \ 	t\ge t_0.
\end{equation}
{ Now let $t=t_0$ in \eqref{subdentro}. We claim that there exists $\mu_0>0$ (depending on $\overline{\varepsilon}$) such that
\begin{equation}\label{subdentro2}
	2t_0^{p-1}D_{p,\gamma, \Omega}\varphi_{1}^{p-r\gamma} \le c_{ \Omega\setminus\overline{\Omega_{\overline{\varepsilon}}}}\mu \stackrel{\eqref{condfn}}{\le} \mu f_n(x)  \text{ in } \ \  \Omega\setminus\overline{\Omega_{\overline{\varepsilon}}}\text{ for all }\mu\geq\mu_0, n\in\mathbb{N}.
\end{equation}
Indeed,} \bk inequality \eqref{subdentro2} amounts  to require that 
\begin{equation*}\label{mu0}
	\displaystyle \mu\ge \mu_0:=\frac{2t_0^{p-1}D_{p,\gamma, \Omega}\|\varphi_{1}\|_{L^\infty(\Omega)}^{p-r\gamma}}{c_{ \Omega\setminus\overline{\Omega_{\overline{\varepsilon}}}}}.
\end{equation*}
We have thus shown that
\begin{equation}\label{equ v fn}
	-\Delta_p v +\frac{a(x)}{v^{\gamma}} \leq \mu f_{n}(x) \ \ \text{ in } \Omega\setminus\overline{\Omega_{\overline{\varepsilon}}} \ \ \text{for} \ \ 	\mu\ge \mu_0\text{ and }t=t_0.
\end{equation}
We then fix $t=t_0$ (i.e. from now on $v=t_0  \varphi_{1}^{r}$) and we  estimate \eqref{e.subsolformula} in $\Omega_{\varepsilon}$ for $0<\varepsilon\leq \overline{\varepsilon}$. Observe that, by \eqref{tzero} and \eqref{phi equiv dist },  $t_0$ depends on $
\varepsilon$ and 
\begin{equation}\label{esplode}
	  t_0(\varepsilon)\geq \frac{c_{3}}{\varepsilon^r} \ \ \quad\text{ for all }\quad 0<\varepsilon\leq \overline{\varepsilon}, 
\end{equation}
for some constant $c_3$ which do not depend on $\varepsilon$.  Now, recall that the Hopf Lemma (see \eqref{phi equiv dist }) guarantees the existence of a constant  $\overline{c}>0$ such that $|\nabla \varphi_{1}|^{p}\geq \overline{c}$ in $\Omega_\varepsilon$ for any $0<\varepsilon\leq \overline{\varepsilon}$.  

Hence, from \eqref{e.subsolformula},  one has 
$$
-\Delta_p v+\frac{a(x)}{v^{\gamma}}\leq-t_{0}^{p-1} \overline{c}C_{p,\gamma}\varphi_{1}^{-r\gamma}+t_{0}^{p-1}D_{p,\gamma,\Omega}\varphi_{1}^{p-r\gamma}+a(x)t_{0}^{-\gamma}\varphi_{1}^{-r\gamma}\text{ in } \Omega_\varepsilon,
$$
which is
\begin{equation}\label{sub2}
	-\Delta_p v+\frac{a(x)}{v^{\gamma}} \leq\left(a(x)t_{0}^{-\gamma}-t_{0}^{p-1} \overline{c}C_{p,\gamma}\right)\varphi_{1}^{-r\gamma}+t_{0}^{p-1}D_{p,\gamma,\Omega}\varphi_{1}^{p-r\gamma}\text{ in }\Omega_\varepsilon.
\end{equation}
Now, assuming $\varepsilon$ small enough, one can deduce
\begin{equation}\label{e.claim1}
	t_{0}^{p-1}D_{p,\gamma,\Omega}\varphi_{1}^{p-r\gamma}\leq\frac{t_{0}^{p-1} \overline{c}C_{p,\gamma}\varphi_{1}^{-r\gamma}}{2}\text{ in }\Omega_\varepsilon.
\end{equation}
Indeed, \eqref{e.claim1} is equivalent to require 
$$
\varphi_{1}^{p}\leq\frac{\overline{c}C_{p,\gamma}}{2 D_{p,\gamma,\Omega}}\text{ in }\Omega_\varepsilon,
$$
which clearly holds if $\varepsilon$ is fixed small enough.

Then, gathering \eqref{e.claim1} into \eqref{sub2}, one gets
$$
-\Delta_p v+\frac{a(x)}{v^{\gamma}} \leq\left(a(x)t_{0}^{-\gamma}-\frac{t_{0}^{p-1} \overline{c} C_{p,\gamma}}{2}\right)\varphi_{1}^{-r\gamma}\text{ in }\Omega_\varepsilon.
$$
Now, using \eqref{esplode}, we get 
$$
-\Delta_p v+\frac{a(x)}{v^{\gamma}} \leq \left(c_{3}^{-\gamma}a(x)\varepsilon^{r\gamma}-\frac{c_{3}^{p-1} \overline{c} C_{p,\gamma}}{2\varepsilon^{r(p-1)}}\right)\varphi_{1}^{-r\gamma}\text{ in }\Omega_\varepsilon.
$$
Therefore, the right-hand of the previous inequality is nonpositive for small enough $\overline{\varepsilon}$  that we fix here, and thus
$$
	-\Delta_p v+\frac{a(x)}{v^{\gamma}}\leq 0\ \ \text{ in }\Omega_{\overline{\varepsilon}}.
$$
Recalling that \eqref{equ v fn} holds in $ \Omega\setminus\overline{\Omega_{\overline{\varepsilon}}}$, we have then proven that 
\begin{equation}\label{subsol}
	-\Delta_p v+\frac{a(x)}{v^{\gamma}}\leq \mu f_n(x) \text{ in }\Omega\text{ for all }\mu\geq\mu_0, 
\end{equation}
for any $n\in\mathbb{N}$. This gives \eqref{subsol0}.

\medskip

 Now we are \bk in position to conclude the proof showing that \eqref{comparison} holds, i.e.  that $$u_{n} \ge v=t_0\varphi_1^{r}\qquad  \text{a.e. in }\,\Omega\qquad \text{for any}\quad n\in\mathbb{N}.$$

We proceed  by induction. For $n=1$,  as \eqref{subsol} is in force, one has 
\begin{equation}\label{subsol2}
	-\Delta_p v+\frac{a(x)}{(v+1)^{\gamma}} \le -\Delta_p v+\frac{a(x)}{v^{\gamma}}\leq \mu f_1(x) = -\Delta_p u_{1}+ \frac{a(x)}{\left(t_0\|\varphi_1^r\|_{L^\infty(\Omega)}+1\right)^{\gamma}} \   \text{ in }\Omega, 
\end{equation}
as, we recall, $u_{0}=t_0\|\varphi_1^r\|_{L^\infty(\Omega)}$. Then, since by definition $u_0 \ge v$  a.e. \bk in $\Omega$, \eqref{subsol2} implies that 
\begin{equation*}\label{subsol3}
	-\Delta_p v \le  -\Delta_p u_{1} \   \text{ in }\Omega, 
\end{equation*}
which, thanks to Lemma \ref{lem comparison}, gives that $u_{1} \ge v$    a.e. \bk in $\Omega$. Now let us assume that $u_{n} \ge v$  a.e. in $\Omega$ \bk and let us prove that $u_{n+1} \ge v$  a.e. \bk in $\Omega$. Once again let us  note that, reasoning as  in \bk \eqref{subsol2}, it holds
\begin{equation*}\label{subsol4}
	\displaystyle -\Delta_p v+\frac{a(x)}{\left(v+\frac{1}{n+1}\right)^{\gamma}} \le \mu f_{n+1}(x) = -\Delta_p u_{n+1}+ \frac{a(x)}{\left(u_{n}+\frac{1}{n+1}\right)^{\gamma}}  \stackrel{u_n \geq v}{\leq} -\Delta_p u_{n+1}+ \frac{a(x)}{\left(v+\frac{1}{n+1}\right)^{\gamma}} \   \text{ in }\Omega, 
\end{equation*}
which implies, using again Lemma \ref{lem comparison},  that $u_{n+1} \ge v$  a.e. \bk in $\Omega$. This shows that $u_{n} \ge v$  a.e. \bk in $\Omega$ for any $n\in\mathbb{N}$.
\end{proof}

\subsection{A priori estimates and convergence results}
Let us collect in the next result  some a priori estimates    for the solutions $u_{n}$  of  \eqref{pbapprox}.

\begin{lemma}\label{lem_priori}
	Let $1<p<N$, $0<\gamma<1$, $0\lneq a\in L^{\infty}(\Omega)$, $\mu >0$,  and let $0\le f\in L^q(\Omega)$, $q\geq 1$. Let $u_{n}$ be a nonnegative solution to \eqref{pbapprox}. Then it holds
	\begin{equation}\label{stimatken}
		\int_{\Omega}|\nabla T_{k}(u_{n})|^{p}\leq\mu k\|f\|_{L^{1}(\Omega)} \ \forall k >0.
	\end{equation}
	Moreover,  it holds:
	\begin{enumerate}
		\item[\textit{i)}] if $q=\frac{Np}{Np-N+p}$ then $u_{n}$ is uniformly bounded in $W^{1,p}_0(\Omega)$ with respect to $n$;
		\item[\textit{ii)}]  if $q=1$ then $u_{n}$ is uniformly bounded in $ M^{\frac{N(p-1)}{N-p}}(\Omega)$ and $|\nabla u_{n}|$ is uniformly bounded in $M^{\frac{N(p-1)}{N-1}}(\Omega)$ with respect to $n$.
	\end{enumerate}
\end{lemma} 
\begin{proof}
The proof is standard once one observes that   $u_n$ is a  subsolution of 
$$
\begin{cases}
		\displaystyle -\Delta_p w_n =\mu f_{n}(x)&\text{ in }\Omega,\\
		w_{n} =  0 &\text{ on }
		\partial\Omega;
\end{cases}
$$
we will briefly sketch it for the sake of completeness. 

\smallskip
Let us take $T_k(u_{n})$ as a test function in the weak formulation of \eqref{pbapprox} in order to deduce that 
	\begin{equation}\label{stimatk}
		\int_{\Omega}|\nabla T_{k}(u_{n})|^{p}\leq\mu k\|f\|_{L^{1}(\Omega)},
	\end{equation}
	where we have got rid of the nonnegative second term on the left-hand of \eqref{pbapprox}. This shows \eqref{stimatken}.	
	\medskip
	
\textit{Proof of i).}  Let us take $u_{n}$ as a test function in the weak formulation of \eqref{pbapprox}, yielding to
 \begin{equation}\label{stima1}
 	\begin{aligned}
 	\int_\Omega |\nabla u_{n}|^p \le \mu\int_\Omega f_n u_{n} &\le \mu \left(\int_\Omega f_n^{q}\right)^{\frac{1}{q}} \left(\int_\Omega u^{\frac{Np}{N-p}}_{n}\right)^{\frac{N-p}{Np}} 
 	\\
 	&\le \mu \mathcal{S}_{p}^{-1} \left(\int_\Omega f_n^{q}\right)^{\frac{1}{q}} \left(\int_\Omega |\nabla u_{n}|^p\right)^{\frac{1}{p}},
 	\end{aligned} 
 \end{equation}
 where we have got rid of the nonnegative second term on the left-hand of \eqref{pbapprox}. We also have applied H\"older's and Sobolev's inequalities and $\mathcal{S}_p$, from here on, is the best constant in the Sobolev inequality for functions in $W^{1,p}_0(\Omega)$. Hence, \eqref{stima1}  gives that $u_{n}$ is bounded in $W^{1,p}_0(\Omega)$ in $n$ as $f_n$ is bounded in $L^q(\Omega)$   (recall that $u_{n}$ has zero Sobolev trace).

\medskip

\textit{Proof of ii).} Then the result standardly follows from \eqref{stimatk} by applying Lemmas $4.1$ and $4.2$ of \cite{B1995}. This concludes the proof.
\end{proof}

Next lemma concerns the identification of the limit of  $u_{n}$ as $n$ tends to $\infty$ in the case $q=1$. In fact, if $q=\frac{Np}{Np-N+p}$, by Lemma \ref{lem_priori} one easily gets the existence of $u\in W^{1,p}_{0}(\Omega)$ such that, up to subsequences $u_{n}\to u$ a.e.  in $\Omega$ and weakly in $W^{1,p}_{0}(\Omega)$. 

\begin{corollary}\label{cor_qo}
	Let $1<p<N$, $0<\gamma<1$, $0\lneq a\in L^{\infty}(\Omega)$, $\mu >0$,  and let $f\in L^1(\Omega)$ satisfy \eqref{e.hypf1}. Let $u_{n}$ be a positive  solution to \eqref{pbapprox}. Then $u_{n}$ converges, up to subsequences,  a.e. \bk in $\Omega$ to a positive  function $u\in M^{\frac{N(p-1)}{N-p}}(\Omega)$ as $n\to\infty$.   
\end{corollary}
\begin{proof}
	As  this proof is quite classical nowadays (after \cite{B1995}), we only present a brief sketch. 
	
	An application of the Sobolev inequality in the left-hand of \eqref{stimatk} allows  us to deduce that \bk
	\begin{equation}\label{levelset0}
	\mathcal{S}_p^{-p}\left(\int_{\Omega}|T_{k}(u_{n})|^{p^*}\right)^{\frac{p}{p^*}}\leq\mu k\|f\|_{L^{1}(\Omega)},\text{ for all }n\in\mathbb{N}\text{ and }k >0.
	\end{equation}
	By denoting $A_{n,k}=\{u_{n}\geq k\}$, the previous  inequality \bk implies  that \bk 
	\begin{equation}\label{level set est v2}
		|A_{n,k}|\leq\left(\frac{\mathcal{S}_p^{p}\mu\|f\|_{L^{1}(\Omega)}}{k^{p-1}}\right)^{\frac{N}{N-p}}.
	\end{equation}
	Now let us show that $u_{n}$ is a Cauchy sequence in measure. Indeed observe that for all $\eta,k>0$ and $l,n\in\mathbb{N}$, we have
	\begin{equation*}\label{equ three v2}
		\{|u_{n}-u_{l}|\geq\eta\}\subset A_{n,k}\cup A_{l,k}\cup\{|T_{k}(u_{n})-T_{k}(u_{l})|>\eta\}.
	\end{equation*}
	 From \bk\eqref{level set est v2}, we know that the sizes of $A_{n,k}$ and $A_{l,k}$ tend to zero as $k\to\infty$. To estimate the size of the third set, we use the fact that, up to subsequences, and as an application of \eqref{stimatk} and Rellich–Kondrachov Theorem, one has that  $T_{k}(u_{n})$ is Cauchy in measure for all $k>0$. Then fixing $k$ large enough, one has that for each $\varepsilon>0$ there exist $n_{\varepsilon}>0$ such that
	$$
	|\{|u_{n}-u_{l}|\geq\eta\}|<\varepsilon\text{ for all }l,n>n_{\varepsilon},
	$$
	which implies that $u_{n}$ is a Cauchy sequence in measure. Thus, recalling Lemma \ref{lemma_controllobasso}, there exists a positive   measurable function $u$ to which $u_{n}$ converges (up to a subsequence)  a.e. \bk in $\Omega$ as $n\to\infty$. 
	From an application of the Fatou Lemma in \eqref{levelset0} and reasoning as above, we get that 	\begin{equation*}\label{e.Marcinone v2}
		|\{u>k\}|\leq \left(\frac{\mathcal{S}_p^{p}\mu\|f\|_{L^{1}(\Omega)}}{k^{p-1}}\right)^{\frac{N}{N-p}}\text{ for all }k>0.
	\end{equation*}
	This proves the result.
\end{proof}

\subsection{Existence for $\mu\geq \mu_{0}$}

We are now in position to prove  Theorem \ref{teo_existence}.

\begin{proof}[Proof of Theorem \ref{teo_existence}]
	
	We fix $\mu\geq \mu_0$ where $\mu_0$ is determined  in \bk Lemma  \ref{lemma_controllobasso}. We aim to  let \bk $n\to\infty$  in \bk the weak formulation of \eqref{pbapprox}. By Corollary \ref{cor_qo} we have that, up to a subsequence, $u_{n}\to u$  a.e. \bk in $\Omega$ as $n\to\infty$. Furthermore both $u_n$ and $u$ are positive thanks to Lemma \ref{lemma_controllobasso}.
	
	\medskip
	
	First we claim that $\mu f_n -(au_{n-1}+\frac{1}{n})^{-\gamma}$ converges in $L^1(\Omega)$ as $n\to\infty$ to $\mu f- au^{-\gamma}$  by applications of the Lebesgue Theorem; indeed it follows from Lemma \ref{lemma_controllobasso}, that for some positive constant $C$, one has 
	\begin{equation*}\label{L1}
		\displaystyle \frac{a(x)}{(u_{n-1}+\frac{1}{n})^\gamma} \stackrel{\eqref{comparison}}{\le} \frac{\|a\|_{L^\infty(\Omega)}}{C\varphi_1^{\frac{p\gamma}{p+\gamma-1}}} \in L^1(\Omega),
	\end{equation*}
	thanks to \eqref{phi equiv dist } and \eqref{distint} since $\gamma<1$.
	
	\medskip	
	
	Furthermore, the  a.e. \bk convergence of $\nabla u_{n}$, up to subsequences, towards $\nabla u$ in $\Omega$ as $n \to \infty$ is in force. Indeed, as $\mu f_n -(au_{n-1}+\frac{1}{n})^{-\gamma}$ converges in $L^1(\Omega)$ as $n\to\infty$,	one can apply  classical stability  results for nonlinear elliptic equations with $L^1$ data (\cite{bm,B1995}) in order to get  this convergence \bk. 
	\medskip
	 
 Thus, we can pass to the limit the principal part of the weak formulation of \eqref{pbapprox}. Indeed Lemma \ref{lem_priori} gives that $|\nabla u_n|^{p-1}$ is bounded in $M^{\frac{N}{N-1}}(\Omega)$ with respect to $n$ and then $|\nabla u_n|^{p-2}\nabla u_n$ converges to $|\nabla u|^{p-2}\nabla u$ in $L^r(\Omega)^N$ with $r<\frac{N}{N-1}$ as $n\to\infty$; this is sufficient to  let \bk $n\to\infty$ in  the \bk term of the weak formulation of \eqref{pbapprox} involving the principal operator.  This  allows us to deduce  that $u$ is a solution to \eqref{e.problem}.
	
	\medskip

	\medskip
	
	The regularity properties of the solution follow by  the \bk lower semicontinuity of the norm with respect to $n$ in the estimates of  Lemma \ref{lem_priori}. This concludes the proof. 
\end{proof}
\begin{proof}[Proof of Theorem \ref{theo pgeqN}] We now turn our attention to the proof of Theorem \ref{theo pgeqN}. Let $p\geq N$ and recall that the following embeddings hold for any $u\in W_0^{1,p}(\Omega)$
$$
\begin{aligned}
& \text{ if }p>N\ \ \ \text{then}\ \ \ \|u\|_{L^{\infty}(\Omega)}\leq c_{\infty} \|u\|_{W_0^{1,p}(\Omega)},\\\
&\text{ if } p=N \ \ \ \text{then}\ \ \ \|u\|_{L^{r}(\Omega)}\leq c_{r}  \|u\|_{W_0^{1,p}(\Omega)},\text{ for all }r>1.
\end{aligned}
$$ 
 We proceed as in the proof of Theorem \ref{teo_existence}. We fix $\mu\geq \mu_0$ where $\mu_0$ is determined in the Lemma  \ref{lemma_controllobasso}. We aim to  let \bk$n\to\infty$  in \bk the weak formulation of \eqref{pbapprox}. Choosing $u_n$ as a test function, and using the injection for $p>N$, we get
$$
\begin{aligned}
	\int_\Omega |\nabla u_{n}|^p \le \mu\int_\Omega f_n u_{n} &\le \mu \|u_n\|_{L^{\infty}(\Omega)}\int_{\Omega}f_n
	\\
	&\le \mu c_\infty \left(\int_\Omega f_n\right)\left(\int_\Omega |\nabla u_{n}|^p\right)^{\frac{1}{p}}.
\end{aligned} 
$$
  A similar estimate can be obtained under the hypothesis $p=N$ and $f\in L^q(\Omega)$ with $q>1$. Consequently, there exists $u\in W_0^{1,p}(\Omega)$ such that $u_n\to u$ a.e. in $\Omega$ and weakly in $W_0^{1,p}(\Omega)$. The result then follows by arguing as in the proof of Theorem \ref{teo_existence}. 
\end{proof}
\subsection{Non-existence for $0<\mu<\mu_{0*}$}

In this section we prove Theorem \ref{t.nonexistence}.

\begin{proof}[Proof of Theorem \ref{t.nonexistence}]
	By contradiction let us assume that there exists a  finite energy solution $u$ to problem  \eqref{e.problem} for any $\mu>0$. Notice in particular that, by \eqref{e.defweaksol},  one should have $u>0$ on $\Omega$. Then, applying Lemma \ref{Lemma_reg}, one can take $v=u$ into \eqref{e.c2}. This yields to
	$$
	\int_{\Omega}|\nabla u|^{p}+\int_{\Omega}au^{1-\gamma}\leq\mu\int_{\Omega}f u.
	$$
	Moreover, as $a\geq c_{0}$, $f\in L^{\infty}(\Omega)$ and,  observing that  $s\leq s^{1-\gamma}+s^{p}$ for all $s\geq0$, one deduces
	$$
	\int_{\Omega}|\nabla u|^{p}+c_{0}\int_{\Omega}u^{1-\gamma}\leq C\mu\int_{\Omega} (u^{1-\gamma}+u^p),
	$$
	where $C=\|f\|_{L^{\infty}(\Omega)}$. If $\mu<\frac{c_{0}}{C}$, then
	$$
	\int_{\Omega}|\nabla u|^{p}\leq C\mu\int_{\Omega}u^p.
	$$
	From \eqref{e.rayleigh}, we obtain
	$$
	(1-C\mu\lambda_{p}^{-1})\int_{\Omega}|\nabla u|^{p}\leq0,
	$$
	which means that if
	$$
	\mu<\min\biggl\{\frac{c_{0}}{C},\frac{1}{C\lambda_{p}^{-1}}\biggl\},
	$$
	then $u\equiv 0$. This provides a contradiction and the proof is concluded.  
\end{proof}

\section{Some compatibility conditions in the case  $\gamma=1$.}
\label{sec:gamma=1}

In this section we briefly comment problem \eqref{e.problem} in the extremal  case $\gamma=1$. By requiring some extra conditions on both $a$ and $f$  one \bk can show the existence of a positive solution to  
\begin{equation}\label{pbgamma1}
	\begin{cases}
		\displaystyle -\Delta_p u+ \frac{a(x)}{u}=\mu f(x) \ &\text{ in }\Omega,\\
		u =  0 \ &\text{ on }
		\partial\Omega,
	\end{cases}
\end{equation}
where $p>1$ and $\mu>0$ is large enough. Let $\overline{\varepsilon}$ a small parameter that will be fixed later; with respect to the previous sections we further assume that  $a\in L^{\infty}(\Omega)$ is a nonnegative function, not identically null,  such that 
\begin{equation}\label{growtha}
	a(x) \le \overline{a}\delta(x)^{\alpha} \ \ \forall x \in \Omega_{\overline{\varepsilon}},
\end{equation}
for some  constants \bk $\overline{a}>0$, $0<\alpha<1$   ($\Omega_{\overline\varepsilon}$ is defined in \eqref{bordo}).
Moreover $0< f \in L^1(\Omega)$ satisfies \eqref{e.hypf1} and the following growth assumption
\begin{equation}\label{growthf}
	f(x) \ge \overline{f} \delta(x)^{-s} \ \ \forall x \in \Omega_{\overline\varepsilon},
\end{equation}
for some $0<\overline{f}\le 1$, $0<s<1$.

\medskip 
Let us state the main existence theorem.

\begin{theorem}\label{teo_existencegamma=1}
	Let $ 1<p<N$, let $0\lneq a\in L^{\infty}(\Omega)$ satisfy \eqref{growtha}, let $f\in L^1(\Omega)$ satisfy both \eqref{e.hypf1} and \eqref{growthf} such that $\alpha+s\geq1$. Then there exists a positive $\mu_0$ such that for any $\mu\ge \mu_0$ there exists a positive solution $u$ to problem \eqref{pbgamma1}\bk.  Furthermore,  $u\in M^{\frac{N(p-1)}{N-p}}(\Omega)$ and $|\nabla u|\in M^{\frac{N(p-1)}{N-1}}(\Omega)$.
\end{theorem}


\begin{proof}[Proof of Theorem \ref{teo_existencegamma=1}]
As in several parts the proof strictly follows the ideas presented in Section \ref{sec_ex<1}, here we present  a brief sketch  focusing on the main differences   with respect to the one of Theorem \ref{teo_existence}.

\medskip

The scheme of approximation is   given again by \eqref{pbapprox}, namely $u_{0}=t_0 \|\varphi_1\|_{L^\infty(\Omega)} \ \text{ in }\Omega$, and for $n>0$
\begin{equation}\label{pbapproxgamma1}
	\begin{cases}
		\displaystyle -\Delta_p u_{n}+ \frac{a(x)}{|u_{n-1}|+\frac{1}{n}}=\mu f_n(x)&\text{ in }\Omega,\\
		u_{n} =  0 &\text{ on }
		\partial\Omega,
	\end{cases}
\end{equation}
where $f_n=T_{n+c_{\Omega\setminus\overline{\Omega_{\overline{\varepsilon}}}}}(f)$. Let us observe that $f_n$ satisfies both \eqref{condfn} and
\begin{equation}\label{growthfn}
f_n(x) \ge \overline{f}  \left(\delta(x)+\frac{1}{n}\right)^{-s} \ \forall x \in \Omega_{\overline\varepsilon}, \ \forall n \in \mathbb{N}.
\end{equation}
The existence of such $u_{n} \in W^{1,p}_0(\Omega)\cap L^\infty(\Omega)$ solution to \eqref{pbapprox} follows again  from \cite{ll} for any $n \in \mathbb{N}$. 

As in Lemma \ref{lemma_controllobasso}, we  aim to show that, for some positive constant $C$, we have  
\begin{equation}\label{confrontogamma=1}
 u_n \ge C\varphi_{1} \ \forall n \in\mathbb{N}.
\end{equation}
Indeed, the main differences rely on showing that in $\Omega$ it holds
\begin{equation}\label{subsol1}
	-\Delta_p v+\frac{a(x)}{v+\frac{1}{n}}\leq \mu f_n(x), 
\end{equation} 
where $v=t\varphi_1$ for some $t>0$. We show it separately on both  $ \Omega\setminus\overline{\Omega_{\overline{\varepsilon}}}$  and  $\Omega_{\overline\varepsilon}$. 
 
We start by  $ \Omega\setminus\overline{\Omega_{\overline{\varepsilon}}}$ where we can assume that
$$\frac{a(x)}{t\varphi_{1}} \le t^{p-1}\lambda_p\varphi_{1}^{p-1},$$
by requiring
\begin{equation*}
	t\ge t_0:=\left(\frac{\|a\|_{L^\infty(\Omega)}}{\displaystyle \lambda_p\min_{ \Omega\setminus\overline{\Omega_{\overline{\varepsilon}}}}(\varphi_{1}^{p})}\right)^{\frac{1}{p}}.
\end{equation*}
  Therefore,   we have
\begin{equation}\label{subdentrogamma=1}
	-\Delta_p v +\frac{a(x)}{v+\frac{1}{n}} \leq 2t_0^{p-1}\lambda_p\varphi_{1}^{p-1} \le c_{ \Omega\setminus\overline{\Omega_{\overline{\varepsilon}}}}\mu \stackrel{\eqref{e.hypf1}}{\le} \mu f_n(x)\text{ in }  \Omega\setminus\overline{\Omega_{\overline{\varepsilon}}},
\end{equation}
where we required that
\begin{equation}\label{mu0gamma1}
	\displaystyle \mu\ge \mu_0:=\frac{2t_0^{p-1}\lambda_p\|\varphi_{1}\|_{L^\infty(\Omega)}^{p-1}}{c_{\Omega\setminus\overline{\Omega_{\overline{\varepsilon}}}}}.
\end{equation}

Now we focus on proving \eqref{subsol1} in $\Omega_\varepsilon$, for $\varepsilon\leq \overline\varepsilon$; thanks to \eqref{growtha}, one has
  \begin{equation}\label{sub2gamma=1}
  	-\Delta_p v+\frac{a(x)}{v+\frac{1}{n}} \leq \frac{\overline{a}\delta(x)^{\alpha}}{t_0\varphi_1+\frac{1}{n}}+t_0^{p-1}\lambda_p\varphi_{1}^{p-1}.
  \end{equation}
  For the right-hand of \eqref{sub2gamma=1}, observe that, as for \eqref{esplode}, by  \eqref{phi equiv dist } one has \begin{equation}\label{esplode1}\frac{c_{3}}{\varepsilon}\leq t_{0}\leq\frac{c_{4}}{\varepsilon}\end{equation} for $c_3,c_4>0$ independent of $\varepsilon$; in particular, without  losing \bk generality, we assume $\varepsilon$ small enough so that 
  \begin{equation}\label{t0c1}
  	t_0c_1\ge 1
  \end{equation}
  where $c_1$  is the constant given into \bk \eqref{phi equiv dist }.

  Using both \eqref{phi equiv dist } and \eqref{growthfn}, and  the fact that $\alpha+s\ge 1$, \bk one has
  \begin{equation*}
  \begin{aligned}
  \frac{\overline{a}\delta(x)^{\alpha}}{t_0\varphi_1+\frac{1}{n}}+t_0^{p-1}\lambda_p\varphi_{1}^{p-1} & \stackrel{\eqref{growthfn}}{\leq}\left(\frac{\overline{a}\delta(x)^{\alpha}}{t_0\varphi_1+\frac{1}{n}}\right)\frac{f_{n}(x)}{\overline{f}\left(\delta(x)+\frac{1}{n}\right)^{-s}}+t_0^{p-1}\lambda_p\varphi_{1}^{p-1}\frac{f_{n}(x)}{\overline{f}\left(\delta(x)+\frac{1}{n}\right)^{-s}} \\
  &\stackrel{\eqref{phi equiv dist }-\eqref{t0c1}}{\leq}\overline{a}\delta(x)^{\alpha}\left(\frac{1}{t_0c_1\delta(x)+\frac{1}{n}}\right)^{1-s}\frac{f_{n}(x)}{\overline{f}}+t_0^{p-1}\lambda_p\varphi_{1}^{p-1}\frac{f_{n}(x)}{\overline{f}\left(\delta(x)+\frac{1}{n}\right)^{-s}}\\
  &\stackrel{\eqref{phi equiv dist }-\eqref{esplode1}}{\le}\left(\frac{\overline{a}}{\overline{f} t_0^{1-s}c_1^{1-s}}\delta(x)^{\alpha+s-1}+\frac{c_4^{p-1}c_{2}^{p-1}\lambda_{p}}{\overline{f}}\left(\delta(x)+\frac{1}{n}\right)^s\right)f_n(x)\\
  &\stackrel{\eqref{esplode1}}{\leq} \tilde{C}\left(\varepsilon^{\alpha}+\left(\varepsilon+\frac{1}{n}\right)^s\right) f_n(x)\text{ in }\Omega_{\varepsilon},\\
 \end{aligned} 
 \end{equation*}
where
$$
\tilde{C}=\frac{\overline{a}}{\overline{f} c_1c_3^{1-s}}+\frac{\lambda_{p}c_{2}^{p-1}c_4^{p-1}}{\overline{f}}.
$$
Thanks to the assumptions on $\alpha$ and $s$, one can fix $\overline{\varepsilon}$ small enough so that
  \begin{equation*}\label{mu0gamma12}
 \tilde{C}\left(\overline{\varepsilon}^{\alpha}+\left(\overline{\varepsilon}+\frac{1}{n}\right)^s\right) \le \mu_0.
  \end{equation*}
   Then, gathering the previous inequalities one has
  \begin{equation}\label{subfuorigamma=1}
  	-\Delta_p v +\frac{a(x)}{v+\frac{1}{n}} \leq  \mu f_n(x)\text{ in } \Omega_{\overline{\varepsilon}},
  \end{equation}
  under the request \eqref{mu0gamma1} (recall from now on $\Omega_{\varepsilon } =\Omega_{\overline{\varepsilon}}$ is fixed). Hence \eqref{subdentrogamma=1} and \eqref{subfuorigamma=1} give that \eqref{subsol1} holds. Once that \eqref{subsol1} is in force, the proof of \eqref{confrontogamma=1} is identical to the one in  Lemma \ref{lemma_controllobasso}.
   
  Now observe that an application of Lemma \ref{lem_priori} (case \textit{ii)}) gives that $u_n$ is uniformly bounded in $M^{\frac{N(p-1)}{N-p}}(\Omega)$ and $|\nabla u_n|$ is uniformly bounded in $M^{\frac{N(p-1)}{N-1}}(\Omega)$. 
  
  From here on the proof is almost identical to the one of the proof of Theorem \ref{teo_existence}. Indeed in order to pass to the limit in the singular term of \eqref{pbapproxgamma1} as $n\to\infty$,  we may  observe that, for some positive constant $C$, one has 
  \begin{equation}\label{app_leb}
  	\frac{a(x)}{u_{n-1} +\frac{1}{n}} \stackrel{\eqref{confrontogamma=1}}{\le} \frac{a(x)}{C\varphi_1} \in L^1(\Omega),
  \end{equation}
  thanks to \eqref{growtha}. Then \eqref{app_leb} allows to apply the Lebesgue Theorem on the singular term as $n\to\infty$. This concludes the proof.
  
\end{proof}

The case of a generic bounded function $a$ in the lower order term  of \eqref{pbgamma1} (i.e. $\alpha=0$ in \eqref{growtha}) is quite exceptional. In this case, it is not generally true that $a{\varphi_1^{-1}}$ belongs to $L^{1}(\Omega)$ and the previous proof does not apply. Surprisingly in this case the existence of a solution of \eqref{pbgamma1} is hampered by the fact that $f$ is too  ``regular" \bk (say summable) near the boundary.

\subsection{Non-existence result for $\gamma=1$}
A non-existence result analogous to the one in Theorem \ref{t.nonexistence} can be proven  for $\gamma =1$ even for an unbounded $f$. We have the following:
\begin{theorem}\label{teononexistencegamma1}
 Let $p>1$, $p^\prime=\frac{p}{p-1}$, assume \bk that $0\leq a\in L^{1}(\Omega)$ is a non-trivial function and that $f\in L^{p'}(\Omega)$. Then there exists $\mu_{0^{*}}>0$ such that problem \eqref{pbgamma1} does not admit any  non-trivial finite energy solution $u\in W_{0}^{1,p}(\Omega)$ for all $0<\mu\leq\mu_{0^{*}}$.
\end{theorem}
\begin{proof}[Proof of Theorem \ref{teononexistencegamma1}] Assume by contradiction that $u\in W_{0}^{1,p}(\Omega)$ is a solution of problem \eqref{pbgamma1}. 
Let
$$
A:=\int_{\Omega}a>0 \qquad \text{and}\qquad  F:=\frac{1}{p'}\int_{\Omega}f^{p'}>0
$$	Since $u\in W_{0}^{1,p}(\Omega)$ and $f\in L^{p'}(\Omega)$, we may apply Lemma \ref{Lemma_reg} to conclude that \eqref{e.c2} holds with $\gamma=1$. Thus, choosing  $v=u$ in \eqref{e.c2}, we get
$$
\int_{\Omega}|\nabla u|^p+A=\mu\int_{\Omega}fu.
$$
We estimate the right-hand side using Young's inequality with exponents $p$ and $p'$. By a further use of Poincar\'e's inequality (recall \eqref{e.rayleigh}), we get
$$
\left(1-\frac{\mu}{p\lambda_p}\right)\int_{\Omega}|\nabla u|^p\,+A\leq\mu F.
$$
Consequently, if
$$
\mu\leq \mu_{0^*}:=\min\biggl\{p\lambda_{p},\frac{A}{F}\biggl\},
$$
one has  $$\int_{\Omega}|\nabla u|^p\,\leq0,$$ and thus $u\equiv 0$. This concludes the proof

\end{proof}
\section*{Declarations:}

\medskip
{\bf Funding:}
F. Oliva and F. Petitta are partially supported by the Gruppo Nazionale per l’Analisi Matematica, la Probabilità e le loro Applicazioni (GNAMPA) of the Istituto Nazionale di Alta Matematica (INdAM). M.F. Stapenhorst is partially  supported by FAPESP 2022/15727-2 and 2021/12773-0.

\medskip
{\bf Conflict of interest declaration:}
The authors declare no competing interests.

\medskip
{\bf Data availability statement:}
 We do not analyse or generate any datasets, because our work
proceeds within a theoretical and mathematical approach. One
can obtain the relevant materials from the references below.  

\medskip
{\bf Ethical Approval:} not applicable

\medskip
 {\bf Acknowledgements:} The authors thank the anonymous referee for the suggestions, which helped to improve the manuscript.\bk

\end{document}